\documentclass[11pt]{amsart}
\usepackage{pstricks,pst-plot, amsmath, amssymb, }
\usepackage{epsfig}
\usepackage{graphics}

\definecolor{Black}{cmyk}{0,0,0,1}
\definecolor{OrangeRed}{cmyk}{0,0.6,1,0}            
\definecolor{DarkBlue}{cmyk}{1,1,0,0.20}
\definecolor{myblue}{rgb}{0.66,0.78,1.00}
\definecolor{Violet}{cmyk}{0.79,0.88,0,0}
\definecolor{Lavender}{cmyk}{0,0.48,0,0}
\parskip=\smallskipamount

\newtheorem{theorem}{Theorem}[section]
\newtheorem{lemma}[theorem]{Lemma}
\newtheorem{corollary}[theorem]{Corollary}

\theoremstyle{definition}

\newtheorem{example}[theorem]{Example}

\newtheorem{remark}[theorem]{Remark}

\newtheorem{defn}[theorem]{Definition}

\newcommand{\C}{\mathbb{C}}

\newcommand{\N}{\mathbb{N}}

\renewcommand\Subset{\subset\subset}

\newcommand{\bea}{\begin{eqnarray*}}
\newcommand{\eea}{\end{eqnarray*}}
\newcommand{\htop}{h_{\operatorname{top}}}

\newcommand{\dist}{{\operatorname{dist}}}

\newcommand{\ra}{\rightarrow}

\newcommand{\diam}{\operatorname{diam}}

\newcommand{\diamE}{\operatorname{\diam_{\operatorname{Eucl}}}}

\numberwithin{equation}{section}
\renewcommand{\epsilon}{\varepsilon}
\renewcommand{\tilde}{\widetilde}
\title{Entropy of Transcendental entire functions}

\author[A.M. Benini]{Anna Miriam Benini$^{\dag}$}
\author[J.E.  Forn{\ae}ss ]{John Erik Forn{\ae}ss$^{*}$}
\author[H. Peters]{Han Peters}



\thanks{$^{\dag}$ This project has received funding from the European Union’s Horizon 2020 research and innovation programme under the Marie Sk{\l}odowska-Curie Grant Agreement No. 703269   COTRADY}
\thanks{$^{*}$  Supported by the NFR grant no. 10445200}
\address{ A.M. Benini: Dipartimento di Scienze Matematiche Fisiche e Informatiche\\
Universit\'a di Parma    \\
 Italy} \email{ambenini@gmail.com}
\address{ H. Peters: Korteweg de Vries Institute for Mathematics\\
University of Amsterdam\\
the Netherlands} \email{hanpeters77@gmail.com}
\address{ J.E. Fornaess: Department of Mathematical Sciences\\
NTNU Trondheim, Norway} \email{john.fornass@ntnu.no}

\date{\today}
\begin{document}

\begin{abstract}
We prove that all entire transcendental entire functions have infinite topological entropy.
\end{abstract}

\maketitle

\section{Introduction}

Topological entropy is a central property in dynamical systems and has been studied extensively, both in the complex setting and outside. More generally it was shown by Misiurewicz and Przytycki \cite{MisiurewiczPrzytycki} that for smooth self-maps of compact manifolds of topological degree $d$ the entropy is at least $\log(d)$. For polynomials and rational functions acting on the Riemann sphere, it was shown independently by Gromov (in a preprint from 1977, published in 2003 \cite{Gromov}) and Lyubich \cite{lyubich} that the topological degree is equal to $\log(d)$.

The goal in this paper is to determine the topological entropy of transcendental entire maps. Such maps have  infinite topological  degree, and hence one can expect that the topological entropy is also infinite. This is indeed the case, as we will prove here.

In \cite{Bergweiler} Bergweiler proved that the Ahlfors Five Islands Property implies for any transcendental function $f$ the existence of a bounded simply connected open set $D \subset \mathbb C$, and disjoint relatively compact subsets $U_1, U_2 \subset \subset D$ which are both being mapped univalently onto $D$ by some iterate $f^k$. As was pointed out by Dujardin in \cite{dujardin}, an immediate consequence is that the topological entropy of a transcendental function is always strictly positive. Since no bound on $k$ is given, the argument does not provide a definite lower bound on the entropy. The fact that the entropy is strictly positive follows also from the results by  \cite{Christensen}.

We will prove the following statement, which gives less information on the way the image covers the domain, but which does imply arbitrarily large lower bounds on the entropy.

\begin{theorem}\label{main}
Let $f$ be a transcendental entire function, and let $N \in \mathbb N$. There exists a non-empty bounded open set $V \subset \mathbb C$ so that $V \subset f(V)$, and such that any point in $V$ has at least $N$ preimages in $V$, counted with multiplicity.
\end{theorem}

The fact that $f$ has infinite entropy follows from the next statement. We refer to the Appendix for the proof and for the definition of  entropy.

\begin{theorem}\label{thm:covering implies entropy}
Let $V\subset \mathbb C$ be a bounded open set, and let $g: V \rightarrow \mathbb C$ be a holomorphic function, having a holomorphic continuation to a neighborhood of $\overline{V}$.
Suppose that every $w \in V$ has at least $N$ preimages in $V$, counted with multiplicity. Then the topological entropy of $g$ is at least $\log(N)$.
\end{theorem}

In the previous paper \cite{BFPViet} we treated the simpler case when the function $f$ omits some value. In this case the domains $V$ can be chosen equal to arbitrarily large annuli of fixed modulus. As will be pointed out in example \ref{example}, this cannot always be done for arbitrary transcendental functions. Instead, the domain $V$ that we construct is either a simply connected subdomain of some annulus, or equals a large disk.

\medskip

\noindent {\bf Acknowledgement:} The fact that transcendental functions have infinite entropy was proved independently by Markus Wendt. His result from 2005, whose proof relies upon Ahlfors Five Island Theorem, was never made public but is mentioned in his PhD thesis \cite[Beispiel 4.7.3]{Wendt}. We are grateful to Walter Bergweiler for bringing the work of Wendt to our attention.

\section{Proof of the main theorem}

\subsection*{Notation}

Throughout the paper we denote by $\Delta(z,r)$ the open Euclidean disk of radius $r>0$ centered at $z \in \mathbb C$. For a set $C\subset\C$ we denote by $\diamE C$ its Euclidean diameter.

For a hyperbolic domain $D\subset \C$ let us denote by $\rho_D (z)|dz|$ its Poincar\'e metric, where $\rho_D(z)$ is the hyperbolic density on $D$. For a subset $D'\subset D$, we denote by $\diam_D(D')$ the diameter of $D'$ in the Poincar\'e metric of $D$. Following \cite{Ahlfors} we will write $\Omega_{0,1}$ for the set $\mathbb C \setminus \{0,1\}$.

\subsection*{Estimates in the hyperbolic metric}
The following estimate on the density of the Poincar\'e metric of the twice puctured domain $\Omega_{0,1}$ is well known, see for example Theorem 1-12 in \cite{Ahlfors}. 

\begin{lemma}\label{lem: hyperbolic metric in two punctured plane}
The hyperbolic density satisfies
$$
\rho_{\Omega_{0,1}}(z)> \frac{1}{2|z|\ln|z|}
$$
for $|z|$ sufficiently large.
\end{lemma}

In fact, more precise estimates by Hempel \cite{Hempel} and Jenkins \cite{Jenkins} show that the above equation holds whenever
$$
\ln|z| \ge K,
$$
where
$$
K = \frac{\Gamma^4(\frac{1}{4})}{4\pi^2} = 4.3768796\ldots,
$$
hence is satisfied when $|z| > e^5$.

From now on we let $D\subset\C$ be a hyperbolic domain, let $d > 0$, and let $C\subset D$   with $\diam_{D}C < d/2$.
 \begin{lemma}\label{lem:one large image for all}
Let $\alpha\in \mathbb C\setminus \{0\}.$ Let $f: D \ra\mathbb C\setminus \{0,\alpha\}$ be holomorphic. Then there exists $k>0$, depending only on $d$, such that the following holds:

If there exists $w_M\in C$ with  $ |f(w_M)| > M > k |\alpha|$, then
$$
|f(z)|>|\alpha|^{\frac{e^d-1}{e^d}} \cdot M^{1/e^d}
$$
for all $z\in C$.
\end{lemma}

\begin{proof}
Let us first suppose that $\alpha=1$. Since $\Omega_{0,1}$ is a complete metric space, for any $d$ there exists $k>0$ such that if $|f(w_M)|>k$ then $f(C)$ is contained in the disk $|z|>e^5$. Since holomorphic maps are distance decreasing
$$
\diam_{f(D)} f(C)< d/2,
$$
and hence
$$
\diam_{\mathbb C \setminus \{0,1\}} f(C) < d/2.
$$
and in particular
$$
\mathrm{dist}_{\mathbb C \setminus \{0,1\}} (f(z), f(w_M)) < d/2
$$
for any $z \in C$.
By Lemma~\ref{lem: hyperbolic metric in two punctured plane} and the fact that $f(C)$ is contained in the disk $|z|>e^5$ it follows that
$$
\begin{aligned}
d/2 & > \dist_{ \C\setminus \{0,1\}}(f(z),f(w_M))\geq\int_{|f(z)|}^{|f(w_M)| }\frac{1}{2t\ln t}\\
&=\frac{1}{2}(\ln \ln |f(w_M)|-\ln \ln |f(z)|),
\end{aligned}
$$
which gives
\bea
|f(z)| > \exp(\exp(\ln \ln |f(w_M|-d))
& = & |f(w_M)|^{1/e^d} > |M|^{1/e^d}\\
\eea

When $\alpha\neq 1$ the result follows directly by considering the function $f(z)/\alpha.$
\end{proof}

From now on we define $k>0$ as in the above lemma, depending on $d$.

\begin{corollary}\label{cor:non-omitted values}
Let $f: D \ra\mathbb C\setminus \{0\}$ be holomorphic, let $w_M\in C$ and write $M = |f(w_M)|$. Let $|\alpha| < M/k$. If there is $z\in C$ so that $|f(z)|\leq |\alpha|^{1-\frac{1}{e^d}}M^{1/e^d}$,
then there exists $z\in D$ so that $f(z)=\alpha.$
\end{corollary}

\begin{proof}
If there is no $z\in D$ so that $f(z)=\alpha,$ then $f:D\rightarrow \C\setminus \{0,\alpha\}$. Moreover $|f(w_M)|=M>k|\alpha|.$
Hence $|f(z)|>|\alpha|^{1-1/e^d}M^{1/e^d}$ for all $z\in C$, a contradiction.
\end{proof}
The following covering lemma bears similarities with Theorem 2.2 in \cite{RS}.
\begin{lemma}\label{thm:covering an annulus}
Let $f: D \ra\mathbb C\setminus \{0\}$ be holomorphic. Let $0 \le m < M$ be such that there exists   $w_M, w_m\in C$ with  $ |f(w_m)|=m$ and  $ |f(w_M)|=M$.
Then $f(D)$ contains the annulus
$$
A=\Big\{\left(\frac{m^{e^d}}{M}\right)^{\frac{1}{e^d-1}}< |z|< M/k\Big\}.
$$
\end{lemma}

\begin{proof}
Let $\alpha \neq 0$ and suppose that $\alpha\notin f(D)$.
By Corollary~\ref{cor:non-omitted values}, if $|\alpha|<M/k$ we have
$$
m\geq |\alpha|^{\frac{e^d-1}{e^d}}M^{1/e^d}
$$
which gives
$$
|\alpha|\leq \left(\frac{m^{e^d}}{M}\right)^{\frac{1}{e^d-1}},
$$
which implies that $\alpha\notin A.$
\end{proof}

From now on we assume that the domain $D\subset \mathbb C$ is simply connected.

\begin{theorem}
Let $f: D \ra\mathbb C\setminus \{0\}$ be holomorphic in a neighborhood of $D$.
Let $0 \le m < M$ be such that there exists   $w_M, w_m\in C$ with  $ |f(w_m)|=m$ and  $ |f(w_M)|=M$.

Let $N \in \mathbb N$ and define
$$
A_N=\Big\{\left(\frac{m^{e^d}}{M}\right)^{\frac{1}{e^d-1}}
< |z|< M/k^N\Big\}.
$$
Then every $\alpha \in A_N$ has at least $N$ distinct preimages in $D$.
\end{theorem}

\begin{proof}
If $A_N$ is empty there is nothing to prove. Otherwise, since $f$ omits $0$ and $D$ is simply connected
we can choose an $N^{th}$-root $g=f^{1/N}$.
Observe that $|g(w_M)|=M^{1/N}$, and that $|g(w_m)|=m^{1/N}$.
Let $\alpha\in A_N$. Let  $\{\eta_j\}_{j=1\ldots N}$ be the  $N$-th roots of $\alpha.$
Let
$$
B=\Big\{\left(\frac{m^{e^d/N}}{|g(w_M)|}\right)^{\frac{1}{e^d-1}}< |z|< |g(w_M)|/k
\Big\}.
$$
Since $\alpha\in A_N$,    $\eta_j\in B$ for all $j$.
By Lemma~\ref{thm:covering an annulus} , for each $j=1\ldots N$ there is $z_j\in D$ so that $g(z_j)=\eta_j.$ By definition of $g$, $f(z_j)=\alpha.$
\end{proof}

The following is immediate, replacing $0$ by any complex number.

\begin{theorem}\label{thm:covering annulus near alpha}
Let $f: {D} \ra\mathbb C\setminus \{\alpha\}$ be holomorphic in a neighborhood of $D$ with $\alpha\in\C$. Let $0<m<M$ be such that there exists   $w_m, w_M\in C$ with $ |f(w_m)|=m$ and $|f(w_M)|=M$. 

Fix $N\in\N$. Let
$$
A_N=\Big\{
\left(\frac{(m+|\alpha|)^{e^d}}{|M-|\alpha||}\right)^{\frac{1}{e^d-1}}
< |z-\alpha|< |M-|\alpha||/k^N\Big\}.
$$
Then every point in  $A_N$ has at least $N$ distinct preimages in $D$.
\end{theorem}

For $R>0$ define the annulus
$$
A_R  :=\{R/2<|z|<2R\}
$$

\begin{corollary}\label{cor: general covering dichotomy new try} 
Let $f:D\ra\C$ be holomorphic in a neighborhood of $D$.
Let $0<m<M$ be such that there exists   $w_m, w_M\in C$ with  $|f(w_m)|=m$ and $|f(w_M)|=M$. Let $k=k(d)$ be as in Lemma~\ref{lem:one large image for all}.

Fix $N \in \mathbb N$, and let $R,j$ such that $k^N< R^{j/2}$. Suppose that $m,M$ satisfy the conditions:
$$
\frac{|M-2R|}{k^N}> 4R
$$
and
$$
\left(\frac{(m+2R)^{e^d}}{|M-2R|}\right)^{\frac{1}{e^d-1}}<\frac{1}{R^{j/2}}.
$$
Then either $A_R \subset f(D)$, or else there exists $\alpha \in A_R\setminus f(D)$ so that
$$
\left(A_R \setminus \Delta(\alpha, \frac{1}{R^{j/2}} )\right) \subset f(D).
$$
In the latter case each $\beta \in A_R \setminus \Delta(\alpha, \frac{1}{R^{j/2}} )$ has at least $N$ distinct  preimages in $D$.
\end{corollary}

\begin{proof}
%
%

If $f(D)\supset A_R$ there is nothing to prove. Otherwise there is $\alpha\in A_R\setminus f(D)$ and we are in the case that $f:D\ra\C\setminus\{\alpha\}$, hence Theorem~\ref{thm:covering annulus near alpha} applies, with $|\alpha|<2R$. In particular $f(D)$ covers at least $N$ times the annulus   $A_N$ defined in Theorem~\ref{thm:covering annulus near alpha}.
The conclusion follows by observing that the conditions on $m,M$ imply  that $\left(A_R\setminus\Delta(\alpha, \frac{1}{R^{j/2}} ) \right)\subset A_N.$
\end{proof}

For $R>0$ and $\theta \in [0, 2\pi]$ we define

\begin{equation*}
\begin{aligned}
D_R & :=\{R/2+1/9<|z|<2R-1/9\,, \; |\mathrm{Arg}(z) - \theta| < 3\pi/4\}, \; \; \mathrm{and}\\
C_R & :=\{2R/3<|z|<3R/2 \, , \; |\mathrm{Arg}(z) - \theta| < 2\pi/3\}.
\end{aligned}
\end{equation*}

Note that $D_R$ is simply connected, that $C_R \subset D_R \subset A_R$, and that $C_R$ has finite hyperbolic diameter in $D_R$, say $d/2$, which is independent from $R$ and $\theta$. From now on we let $k>0$ be the corresponding constant found in Lemma~\ref{lem:one large image for all}.

\begin{theorem}\label{thm:covering dichotomy}
Let $f$ be a transcendental entire function. Let $N\in\N$.
Then there exist arbitrarily large $R$ and $j$ large and $\theta \in [0, 2\pi]$ so that either $A_R \subset f(D_R)$ or else there exists  $\alpha \in A_R\setminus f(D_R)$  so that $\left(A_R \setminus \Delta(\alpha, \frac{1}{R^{j/2}} )\right) \subset f(D_R)$.  In the latter case, each $\beta\in\left(A_R \setminus \Delta(\alpha, \frac{1}{R^{j/2}} )\right) $ has at least $N$ distinct preimages in $D_R$.
\end{theorem}

\begin{remark}\label{rem:position of wi}
We can moreover guarantee that $C_R$ contains at least two points of maximum modulus which are at least $R/10$ apart from each other as well as from the boundary of $C_R$,  and that there is a point $w_m\in C_R$ with $|f(w_m)|<3$ and whose distance from $\partial C_R$  is at least $R/10$.
\end{remark}

\begin{proof}[Proof of Theorem~\ref{thm:covering dichotomy} and Remark~\ref{rem:position of wi}]
Observe that the hypotheses on $m$ and $M$ in Corollary~\ref{cor: general covering dichotomy new try}  are satisfied provided that there exists $w_m, w_M\in C_R$ such that   $|f(w_M)|=M>R^j$ and $|f(w_m)|=m<3R$ for large enough $R,j$. Since the maximum modulus of $f$ on $\{|z| = R\}$ grows faster than any polynomial in $R$, for $R$ large enough we can always assume that there is a point $w_M$ with $|w_M| = R$ and $|f(w_M)|>R^j$.

By Picard's Theorem, $f$ takes on every value infinitely many times except at most one value, so we can choose arbitrarily large $R$ so that there also exists a point $w_m$ with $|w_m| = R$ and $|f(w_m)|= m < 3 R$.
Since $C_R$ contains strictly more than half the circle $|w| = R$, and that there are points of maximum modulus for every $R$, it follows that we can choose $\theta$ so that both $w_m$ and $w_M$ are contained in $C_R$, and such  that $C_R$ contains at least two points of maximum modulus which are at least $R/10$ apart from each other as well as from the boundary of $C_R$.
The claim follows from Corollary~\ref{cor: general covering dichotomy new try}.
\end{proof}

\begin{example}\label{example} It is not true in general that for any entire transcendental function $f$ there exists $R>0$ such that $f(A_R)$ covers $A_R$ arbitrarily many times. Indeed, let
\[
f(z)=\prod_{i=1}^\infty \frac{z_i-z}{z_i},
\]
where $z_i\rightarrow \infty$ very rapidly. For any $R$ the set $f(A_R)$ covers $A_R$ at most once. To see this, notice that for $|z_{j-1}| << |z| << |z_{j+1}|$ one obtains
$$
f(z) \sim c_j z^{j-1} \cdot \frac{z_j - z}{z_j},
$$
where $c_j=(z_0\cdot z_1\cdot\ldots \cdot z_{j-1})^{-1}$.
Let $w \in A_R$. By Rouch\'e's Theorem the difference between the number of solutions to the equation $f(z) = w$ on the two disks $\Delta(0,R/2)$ and $\Delta(0,2R)$ is at most $1$, hence $f(A_R)$ covers $A_R$ at most once.

We note however that this example does have infinite entropy. Indeed, for $|z_i| << R << |z_{i+1}|$ consider the image of the disk $\Delta(0,R)$. By the above estimates each point in $\Delta(0,R)$ will have exactly $i$ preimages in $\Delta(0,R)$, counted with multiplicity. By Theorem \ref{thm:covering implies entropy} the entropy of $f$ is at least $\log(i)$.
\end{example}

\subsection*{Final preparations}

We will make a few elementary observations before we start the proof of our main result.

\begin{lemma}\label{lem:proper maps} Let $N \in \mathbb N$, and let $D\subset \mathbb C$ be a bounded simply connected domain, let $f$ be a holomorphic function defined in a neighborhood of $\overline{D}$, and suppose that there exists $r>0$ such that $|f(z)|\ge r$ for all $z\in\partial D$.
If there exists $\xi\in \Delta(0,r)$ with $N$ preimages in $D$, then every point in $\Delta(0,r)$ has $N$ preimages in $D$ (counted with multiplicity).
\end{lemma}
\begin{proof}
For  $w\in\Delta(0,r)$ let  $g_w=f-w$.  We claim that $g_w$ has the same number of zeroes (counted with multiplicity) as $f$. Observe that  $|g_w-f|<|f|$ on $\partial D$  because $|w|<r$ and $|f|\ge r$ on $\partial D$. The claim follows by Rouch\'e's Theorem since the function $g_{\xi}$ has $N$ zeroes.
\end{proof}

\begin{lemma}\label{lem:lemma for the first case} Let $R,j>4$. There exists $d>0$ such that the following holds. Let $z_1,z_2\in C_R$, $ \alpha\in A_R$,   and assume that $z_i\notin\Delta(\alpha,R/20)$ for $i =1, 2$. Then there exists a simply connected open set  $D\subset D_R\setminus \Delta(\alpha, 1/R^{j/2})$ with $z_1, z_2 \in D$ such that $\dist_D(z_1,z_2) < d/2$.
\end{lemma}
 \begin{proof}
We consider three cases: \\
(i) $\Delta(\alpha, R/20)\cap D_R=\emptyset.$ Then we choose $D=D_R.$\\
(ii) $\Delta(\alpha, R/20) \cap C_R=\emptyset.$
Let $D$ be the tubular neighborhood of $C_R$ with radius $R/20.$\\
(iii) $\Delta(\alpha,R/20) \cap C_R$ is nonempty. Let $I_1\ldots I_4$ be   four arcs starting at $\alpha,$ two radial segments and two circular arcs, ending when they hit the boundary of $D_R$ (see Figure \ref{fig:rays} for an illustration). Then we let $D=D_R\setminus (\Delta(\alpha,R/20)\cup I_i)$ for a suitable $i$ depending on the position of the points $z_1, z_2$. It is clear that $D$ is simply connected, and that $i$ can be chosen to obtain a uniform bound on $\dist_D(z_1,z_2)$ not depending on the positions of $z_1, z_2$ and of $\alpha$.
\end{proof}

 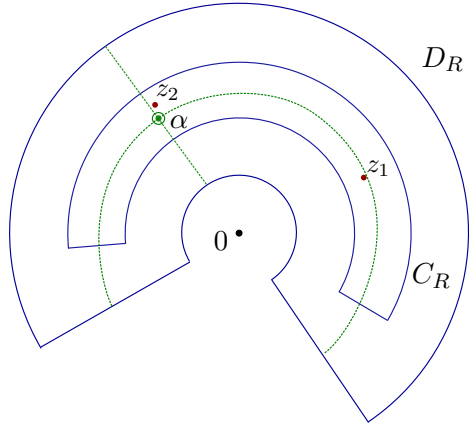
\begin{figure}[hbt!]
\begin{center}
\def\svgwidth{0.5\textwidth}
\begingroup%
  \makeatletter%
   \setlength{\unitlength}{\svgwidth}%
  \makeatother%
  \begin{picture}(1,0.88338011)%
    \put(0,0){\includegraphics[width=\unitlength]{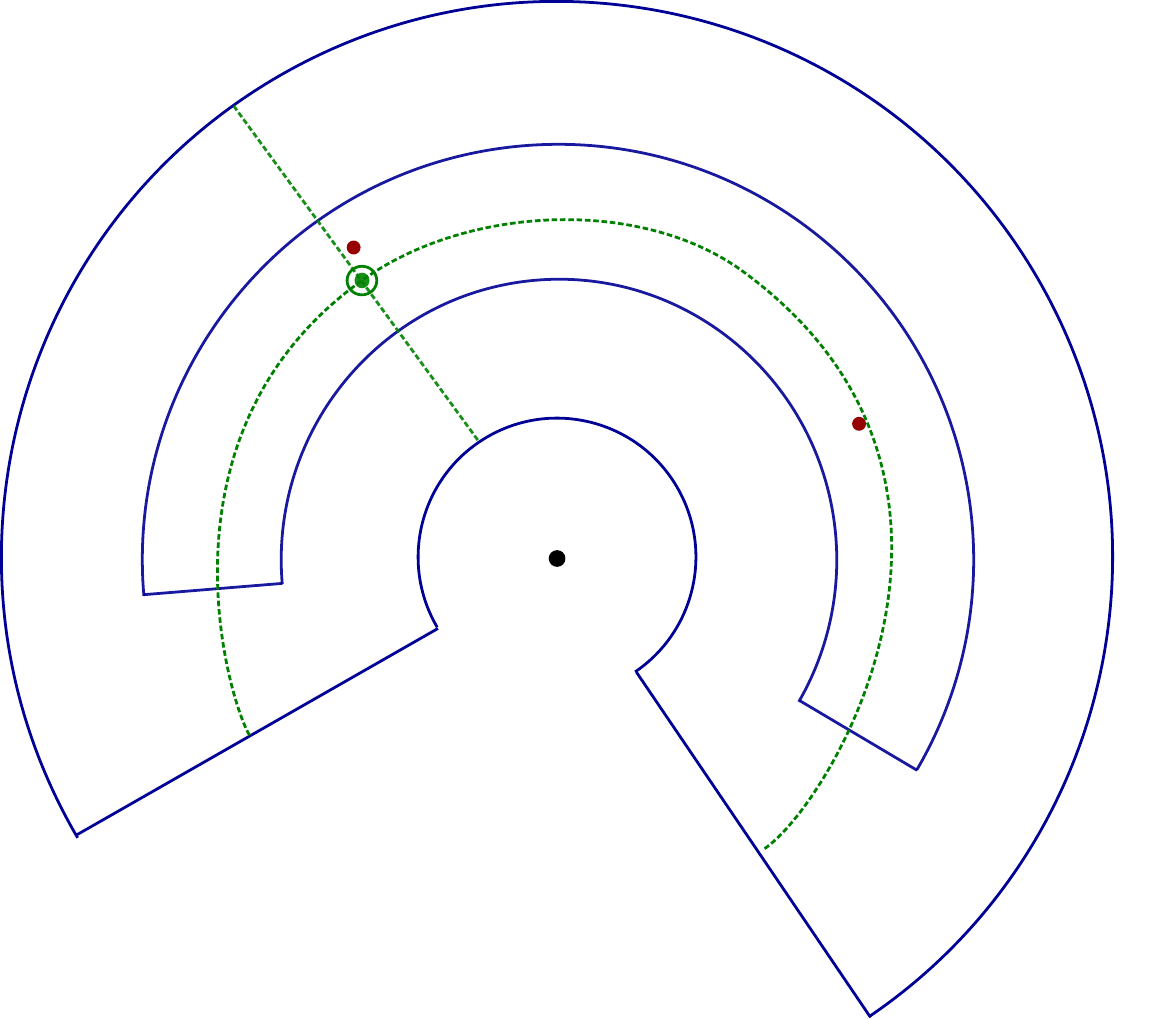}}%
    \put(0.43,0.36390268){\color[rgb]{0,0,0}\makebox(0,0)[lb]{{0}}}%
    \put(0.30810753,0.67805137){\color[rgb]{0,0,0}\makebox(0,0)[lb]{\small{$z_2$}}}%
    \put(0.75083728,0.51880158){\color[rgb]{0,0,0}\makebox(0,0)[lb]{\small{$z_1$}}}%
    \put(0.33774402,0.61959011){\color[rgb]{0,0,0}\makebox(0,0)[lb]{{$\alpha$}}}%
    \put(0.86690218,0.7391595){\color[rgb]{0,0,0}\makebox(0,0)[lb]{{$D_R$}}}%
    \put(0.84654992,0.28377819){\color[rgb]{0,0,0}\makebox(0,0)[lb]{{$C_R$}}}%
  \end{picture}%
\endgroup%
\end{center}
\caption{Illustration of the proof of Lemma~\ref{lem:lemma for the first case}.  In green: the point $\alpha$, the boundary of the disk $\Delta(\alpha, R/20)$ and the four arcs $I_i$.}
\label{fig:rays}
\end{figure}

\begin{lemma}\label{lem:lemma for the second case}
Let $\epsilon >0$ and $\ell \in \mathbb N$. Let $\alpha, z_1, z_2$ and $x_1, \ldots , x_{\ell}$ be points in  the annulus
$$A_R(\epsilon):=\{ R/2+\epsilon  R  \leq|z|\leq 2R-\epsilon  R   \}   $$
and assume that $|\alpha - z_j|$ and $|x_i - z_j| \ge \epsilon R$ for all $i = 1, \ldots \ell$ and $j = 1,2$.
Then there exists $d>0$, depending only on $\epsilon$ and $\ell$, and a simply connected domain $D \subset A_R$ avoiding all the points $x_i$ and satisfying
$$
D \cap \Delta(\alpha, \frac{1}{R^{j/2}}) = \emptyset,
$$
such that $\dist_D(z_1,z_2)<d/2$.
\end{lemma}

\begin{proof}
Up to rescaling we may assume that $R=1$.
Observe that for each choice of points $\alpha, z_1, z_2$ and $x_1, \ldots , x_{\ell}$ we can find such a simply connected domain $D$ containing $z_1, z_2$ by removing the disk $\Delta(\alpha, \frac{1}{R^{j/2}})$ and for each point $\alpha$ or $x_i$ a path connecting the point $\alpha$ or $x_i$ to $\partial A_R$. Each path can be chosen to be either a radial interval, or a combination of a small circular interval and a radial interval.

Note that the construction also works when the points lie in the closed annulus $\overline{A_R(\epsilon)}$
and that each construction gives uniform estimates on the hyperbolic distance between $z_1$ and $z_2$ for nearby locations of the points. Compactness of the initial conditions implies that the constant $d$ depends only on $\epsilon$ and $\ell$.

\end{proof}

\subsection*{Main statement and proof}

Let us recall the statement of our main Theorem:

\medskip

\noindent{\bf Theorem \ref{main}.} \emph{Let $f$ be a transcendental entire function, and let $N \in \mathbb N$. There exists a non-empty bounded open set $V \subset \mathbb C$ so that $V \subset f(V)$ and  such that any point in $V$ has at least $N$ preimages in $V$ under $f$, counted with multiplicity.
}

\medskip

\noindent \emph{Proof.}
Fix $N \in \mathbb N$.
Let $d/2$  be such that   Lemma~\ref{lem:lemma for the first case} and Lemma~\ref{lem:lemma for the second case} hold  for  $\ell = N$ and  $\epsilon = \frac{R}{2N(N+2)}$.
Observe that if Lemma~\ref{lem:lemma for the second case} is satisfied for   $d/2$ with $\ell=n$, it is also satisfied  for all $\ell<n$. Let $k$ be so that Lemma~\ref{lem:one large image for all} holds for  $d/2$. Let $j$ and $R$ large enough so that Corollary~\ref{cor: general covering dichotomy new try} holds for $k$. Choose $R, \theta,$ and $j$ such that the hypotheses in Theorem~\ref{thm:covering dichotomy} are satisfied.

It follows that  either $A_R \subset f(D_R)$, or  $f(D_R)$ covers $A_R\setminus \Delta(\alpha,1/R^{j/2})$ at least $N$ times  for some   $\alpha\in A_R\setminus f(D_R)$.

\medskip

\noindent{\bf Case I. : $f(D_R)\not\supset A_R$.}

In this case there exists $\alpha\in A_R\setminus f(D_R)$ such that  $f({D}_R)$ covers at least $N$ times the set $A_R  \setminus \Delta(\alpha,1/R^{j/2})$. Let us look for a subset of ${D}_R$ such that its image covers itself at least $N$ times.

If
\begin{equation}\label{eqtn:nonrecurrence}
f(\Delta(\alpha,1/R^{j/2})\cap D_R)\cap D_R = \emptyset,
\end{equation}
we can choose $V=D_R\setminus \Delta(\alpha, 1/R^{j/2})$ and
the proof is complete.
Hence we can assume that $f(\Delta(\alpha,1/R^{j/2})\cap D_R)\cap D_R \neq \emptyset$, and in particular that there exists a point $\xi \in \Delta(\alpha,1/R^{j/2})\cap D_R$ with $|f(\xi)|< 2R.$

We will also assume that $\Delta(\alpha,R/20)\Subset D_R$. Indeed, if this is not the case, the proof is completely analogues by replacing $D_R$ by a slightly larger  simply connected open set $\tilde{D}_R \subset \subset A_R$ for which $\Delta(\alpha,R/20) \subset \subset \tilde{D}_R$ is satisfied. In this case, if $f(\tilde{D}_R)$ keeps omitting $\alpha$ we apply the proof of case I, otherwise we move to case II.

Let  $w_M$ be a point in $C_R \setminus\Delta(\alpha, R/20)$ for which $|f(w_M)|\geq R^j $.
Recall that we may assume that such point exists, since by Remark~\ref{rem:position of wi} we can choose $C_R$ to contain at least two points of maximum modulus of distance at least $R/10$ apart from each other.

We claim that
there also exists a point $w_m\in C_R\setminus \Delta(\alpha,R/20)$ so that  $|f(w_m)|<3R.$
Let $w_m\in C_R$ be as in Remark~\ref{rem:position of wi}. If $\Delta(\alpha, R/20)\cap\partial C_R\neq\emptyset$, $w_m\in C_R\setminus \Delta(\alpha,R/20)$ as required. Otherwise $\Delta(\alpha, R/20)\Subset C_R$. In this case, let us assume by contradiction that
 $|f(z)|>3R$ for all $z\in C_R \setminus \Delta(\alpha,R/20)$. Then we also have that $|f(z)|\ge 3R$ on $\partial \Delta(\alpha, R/20)$.
By Lemma~\ref{lem:proper maps},  since there is $\xi\in  \Delta(\alpha,R/20) $ with  $|f(\xi)|<2R$,
we have that $f(\Delta(\alpha, R/20))\supset\Delta(0,3R)$.
This  contradicts the fact that  $\alpha\in \Delta(0,3R)$ was assumed not to lie in $f(D_R)$.

Now let $D$ be as in  Lemma~\ref{lem:lemma for the first case}, where $z_1 := w_M$ and $z_2 := w_m$. Since $A_R$ is not contained in $f(D)$, it follows by Corollary~\ref{cor: general covering dichotomy new try} that $f(D)$ covers $A_R\setminus \Delta(\alpha, R^{j/2})$ at least $N$ times. Since $D$ is contained in $A_R\setminus \Delta(\alpha, R^{j/2})$ this concludes the proof of case I.

\medskip

\noindent{\bf Case II: $f({D}_R)\supset A_R$.}  

Observe that for each fixed $N$, Theorem~\ref{thm:covering dichotomy} holds for  arbitrarily large radii $R$. If there is at least one of them for which case I holds, we are done. Otherwise, for every $R$ given by Theorem~\ref{thm:covering dichotomy} we have that $f(D_R)\supset A_R$ and hence that $f(A_R)\supset A_R$.

If there are arbitrarily large $R$ for which  $f(A_R)$ covers itself  at least $N$ times we are also done.  Hence we may assume that there exists $1 \le \ell < N$ such that for any of the $R$ given by Theorem~\ref{thm:covering dichotomy} there is a point $\alpha = \alpha(R)\in A_R$ which has at most $\ell$ preimages in $A_R$, counted with multiplicity. We can therefore find a sequence of values of $R$ for which the maximum number of preimages in $A_R$ of some point $\alpha$ is at most  $\ell$, and write $\zeta_1=\zeta_1(R),\dots, \zeta_\ell=\zeta_\ell(R) \in A_R $ for the preimages of $\alpha$ in $A_R$.

Let $ W:=A_R\cap\{z\in\C: |f(z)|<2R\}$. For $i=1\ldots \ell$ let $W_i$ be the connected component of $W$ which contains $\zeta_i$ (possibly, they are not all distinct). Now one of the two following cases occurs.

\medskip

\noindent {\bf Case IIa: For arbitrarily large $R$ there exists $R/2<r<2R$ such that the circle $\{|z| = r\}$ does not intersect the set $W$.}

We claim that if $R$ is chosen large enough then  $f(\Delta (0,r))$ covers $\Delta (0,r)$ at least $N$ times, giving the claim.
Let $v\in \mathbb C$ be a non-exceptional value for $f$ with $|v|<1$.
By Picard's theorem, $f$ takes on the value $v$ infinitely many times in any neighborhood of infinity, hence by choosing $R$ sufficiently large  we may assume that $v$ has at least $N$ preimages in the disk $\Delta (0,R/2)\subset \Delta(0,r)$.
Since $W$ does not intersect the circle $\partial\Delta(0,r)$, we have that $|f(z)|\geq 2R$ on $\partial\Delta(0,r)$.
Hence by Lemma~\ref{lem:proper maps}, in $\Delta (0,r)$, the function $f$ takes on any value in  $\Delta (0,2R)\supset\Delta(0,r)$ at least $N$ times, counted with multiplicity.

\medskip

\noindent {\bf Case IIb: for arbitrarily large $R$ the set $W$ intersects all circles $\{|z| = r\}$ for $R/2<r<2R$.}

Then there is some $W_i$, say $W_0$ up to relabeling, with  diameter at least $\frac{3R}{2\ell}$ for arbitrarily large $R$.

We claim that there exist $w_m,w_M\in A_R$ with $|f(w_m)|<2R$ and $|f(w_M)|>R^j, $ and such that $|w_m-\zeta_i|, |w_M-\zeta_i|>\frac{R}{2\ell(\ell+2)}$ for $i=1,\ldots, \ell$, and $|w_m-\alpha |, |w_M-\alpha |>\frac{R}{2\ell(\ell+2)}$.  We also claim  that the distance between $w_m,\ w_M$ and the boundary of $A_R$ is at least $\frac{R}{2\ell(\ell+2)}$.

Indeed, there are at most $\ell+1$ points in $W_0$ that need to be avoided (all of the $\zeta_i$ and $\alpha$), so   we can always find $w_m\in W_0$ which is  at Euclidean distance at least $\frac{\diamE W_0}{2(\ell+2)}>\frac{3R}{4\ell(\ell+2)}$ from all of the $\zeta_i$   and from $\alpha$,   as well as from the boundary of $A_R$.  By definition $|f(w_m)|<2R$.

To find $w_M$ it is enough to find a point of maximum modulus in $A_R$  minus  the set $U=\bigcup_i \Delta(\zeta_i, \frac{R}{2\ell(\ell+2)} )\cup \Delta(\alpha, \frac{R}{2\ell(\ell+2)})$, and which is at distance at least $\frac{R}{2\ell(\ell+2)}$ from $\partial A_R$. This means that  we have to avoid at most $\ell+2$ disks of diameter $ \frac{R}{\ell(\ell+2)}$, hence there are circles in  $A_R\setminus U$ in which we can choose a point of maximum modulus as required, which settles the claim.

By Lemma~\ref{lem:lemma for the second case} and our choice of $d$ in the beginning of the proof, we can find $D\subset A_R$ simply connected with $w_m,w_M\in D$ and $\zeta_i \notin D$ for $i = 1, \ldots \ell$, and with $D \cap \Delta(\alpha, \frac{1}{R^{j/2}}) = \emptyset$, and such that $\dist_D(w_m,w_M)<d/2$. By our choice of the constants $k$, $j$ and $R$ Corollary~\ref{cor: general covering dichotomy new try} holds, and since $f(D)$ omits $\alpha$ by construction, we have that    for $R$ sufficiently large  $f(D)$ covers at least $N$ times the set
$$
A_R\setminus \Delta(\alpha,1/ R^{j/2})\supset D.
$$

\hfill $\square$

Let us observe that each of the three cases I, IIa and IIb can occur. Indeed, case I occurs when $f$ has an omitted value \cite{BFPViet}; case IIa occurs in Example~\ref{example}; and case IIb occurs for $f(z)=e^z$.

\section*{Appendix: Topological entropy on $\mathbb C$.}\label{sect:entropy}

For maps acting on compact spaces the concept of topological entropy has been introduced in \cite{Adler}. In the literature there are several non-equivalent  natural generalizations for the definition of topological entropy on non-compact spaces (see for example
\cite{Bowen}, \cite{Bowen71}, \cite{Bowen73}, \cite{Hofer}, and more recently \cite{Hasselblatt}). We will use the following:

\begin{defn}[Definition of topological entropy]\label{defn:entropy noncompact}
Let $f:Y \rightarrow Y$ be a self-map of a metric space $(Y, d)$.
Let $X$ be a compact subset of $Y.$ Let   $n \in \mathbb N$ and $\delta > 0$.
A  set $E\subset X$ is called \emph{$(n,\delta)$-separated}  if \begin{itemize}
\item for any  $z\in E$, its orbit $\{z, f(z),\dots, f^{n-1}(z)\}\subset X$;
\item  for any $z\neq w\in E$ there exists $k\leq n-1$ such that
$d(f^k(z),f^k(w))> \delta$.
\end{itemize}
Let $K(n, \delta)$ be the maximal cardinality of an $(n,\delta)$-separated set.
Then the \emph{topological entropy} $\htop(X,f)$ is defined as
$$
\htop(X,f):=\sup_{\delta>0}\left\{\limsup_{n\ra\infty}\frac{1}{n}\log K(n,\delta)\right\}.
$$
We define the topological entropy  $\htop(f)$ of $f$ on $Y$ as the supremum of $\htop(X,f)$ over all compact subsets $X\subset Y.$
\end{defn}

When $Y$ is compact the definition coincides with the usual definition. In the literature the finite orbits  $\{z, f(z),\dots, f^{n-1}(z)\}$ are often not required to remain in $X$. A disadvantage of this definition is that the entropy is then dependent on the metric; for example, the entropy of a polynomial acting on the complex plane is then infinite with respect to the Euclidean metric. Our definition above, which may give a smaller value for the entropy, is independent of the metric inducing the topology, and is invariant under topological conjugation.

We are now ready to prove Theorem~\ref{thm:covering implies entropy}. The ideas of the proof are similar to the ideas used in \cite{MisiurewiczPrzytycki}.

\begin{proof}[Proof of Theorem~\ref{thm:covering implies entropy}] Denote the set of critical points of $g$ in $\overline{V}$ by $\mathcal{C}$. Note that $\mathcal{C}$ is finite. Let $\mathcal{C}_0 \subset \mathcal{C}$ contain only those critical points that are not periodic. Write $D$ for the product of the local degrees of $g$ at the critical points in $\mathcal{C}_0$.

Fix a point $w \in V$, not contained in a periodic cycle containing a critical point. It follows that all inverse orbits of $w$ avoid a sufficiently small neighborhood of each super-attracting periodic cycle. Let us denote the complement of these neighborhoods in $V$ by $V'$.

Let $m \in \mathbb N$, and let $\rho=\rho(m)>0$ be such that for every $x \in \mathcal{C}_0$ and every $n = 1, \ldots , m$ we have 
$$
g^n(\Delta(x,\rho)) \cap \Delta(x,\rho) = \emptyset.
$$
Such $\rho$ can be chosen by finiteness of $\mathcal{C}_0$, and since the points $x$ are not periodic. By decreasing $\rho>0$ if necessary we may assume that the disks $\Delta(x,\rho)$ are pairwise disjoint. 

There exists an $\epsilon=\epsilon(m)>0$ such that the following two properties hold: For each $y \in V' \setminus \bigcup_{x \in \mathcal{C}_0} g (\Delta(x, \rho))$ there are at least $N$ preimages of $y$ that are $\epsilon$-separated. 
 On the other hand, if $y \in g(\Delta(x, \rho))$ then the number of preimages (counted with multiplicity) near $x$ that are not $\epsilon$-separated is at most the local degree of $g$ at $x$, and the other preimages have distance at least $\epsilon$ to the preimages near $x$.

Consider a finite inverse orbit $y_0, y_{-1}, y_{-m}$ of a point $y_0 \in V'$. By the estimates on the number of preimages that may not be separated, and by the fact that any inverse orbit of length $m$ enters each disk $\Delta(x,\rho)$ at most once, it follows that there are at most $D-1$ other inverse orbits of $y_0$ of length $m$ that are not $\epsilon$-separated from $y_0$. Thus, a lower bound for the number of $\epsilon$-separated backwards $m$-orbits of $y_0$ is given by $\frac{N^m}{D}$.

Since the lower estimate holds for any $y \in V'$, it holds in particular for any point in $f^{-km}(w)$. Hence for any $k\in \mathbb N$, the number of $\epsilon$-separated backwards orbits of $w$ of length $km$ is at least
$$
\left( \frac{N^m}{D}\right)^k = \left( \frac{N}{D^{1/m}}\right)^{km},
$$
which is therefore a lower bound for $K(km,\epsilon)$, the maximal cardinality of a $(km,\epsilon)$-separated set. 
Since $D$ is a fixed constant and we can let $m$ converge to infinity as $\epsilon \rightarrow 0$, it follows that the topological entropy is at least $\mathrm{log}(N)$.
\end{proof}
\bibliographystyle{amsalpha}
\bibliography{entropy}

\end{document}